\documentclass{article}
\usepackage{amsthm}
\usepackage{amsmath}
\usepackage{amssymb}
\usepackage{titlesec}
\usepackage{hyperref}

\newtheorem{theorem}{\sc{Theorem:}}[section]

\newtheorem{definition}[theorem]{\sc{Definition:}}

\newtheorem{lemma}[theorem]{\sc{Lemma:}}
\newtheorem{proposition}[theorem]{\sc{Proposition:}}

\titleformat{\section}
  {\normalfont\large\sc\centering}{\thesection.}{0.5em}{}
\titlespacing{\section}{0pt}{2pc}{1pc}

\theoremstyle{remark}
\newtheorem{remark}[theorem]{\sc{Remark:}}

\renewenvironment{abstract}{\textbf{\abstractname:}}{\vspace*{\baselineskip}}

\begin{document}

\title{\sc{\textbf{\large THE ATTENUATED MAGNETIC RAY TRANSFORM ON SURFACES}}}
\date{}
\author{\small GARETH AINSWORTH}
\maketitle
\begin{abstract} It has been shown in \cite{paternain111} that on a simple, compact Riemannian 2-manifold the attenuated geodesic ray transform, with attenuation given by a connection and Higgs field, is injective on functions and 1-forms modulo the natural obstruction. Furthermore, the scattering relation determines the connection and Higgs field modulo a gauge transformation. We extend the results obtained therein to the case of magnetic geodesics. In addition, we provide an application to tensor tomography in the magnetic setting, along the lines of \cite{paternain211}.
\end{abstract}

\section{Introduction}
\subsection{\normalsize Magnetic Flows}
\indent \indent Let \((M,g)\) be a compact oriented Riemannian manifold with boundary. Consider the function \(H:TM\rightarrow\mathbb{R}\) given by 
\[H(x,v):=\frac{1}{2}g(v,v),\ \ \ (x,v)\in TM.\]
The geodesic flow on \(TM\) is given by the Hamiltonian flow of the above function with reference to the symplectic structure \(\omega_{0}\) on \(TM\) provided by the pullback, via the metric, of the canonical symplectic form on \(T^{*}M.\) The abstract formulation of a magnetic field imposed on \(M\) is specified by a closed 2-form \(\Omega\). We say that the \textit{magnetic flow}, or twisted geodesic flow, is defined as the Hamiltonian flow of \(H\) under the symplectic form \(\omega\), where \[\omega:=\omega_{0}+\pi^{*}\Omega,\] and \(\pi:TM\rightarrow M\) is the usual projection. Magnetic flows were first studied in \cite{anosov67, arnold61}, and more recently in relation to inverse problems in \cite{dairbekov107, herreros10, herreros10a}.\\
\indent We may alternatively think of the magnetic field as being determined by the unique bundle map \(Y:TM\rightarrow TM,\) defined via, \[\Omega_{x}(\xi,\eta)=g(Y_{x}(\xi),\eta),\ \ \ \forall x\in M,\ \forall \xi,\eta\in T_{x}M.\] Note that this implies that \(Y\) will be skew-symmetric.
The advantage of this point of view is that it provides a nice description of the generator of the magnetic flow, indeed, one can show that this vector field at \((x,v)\in TM\) is given by \[X(x,v)+Y^{i}_{k}(x)v^{i}\frac{\partial}{\partial v^{k}}.\] Here note that the coefficients of \(Y\) are given by \(Y(\frac{\partial}{\partial x^{j}})=Y^{j}_{i}\frac{\partial}{\partial x^{i}}\), and \(X\) denotes the geodesic vector field. Furthermore, the magnetic geodesics, that is the projection of the integral curves of the above vector field to \(M\), are precisely the solutions \(t\mapsto (\gamma(t),\dot{\gamma}(t))\) to the following equation: \[\nabla_{\dot{\gamma}} \dot{\gamma}=Y(\dot{\gamma}).\]
\indent Magnetic geodesics preserve \(H\), and thus have constant speed. In what follows we will restrict ourselves to working on the unit tangent bundle \(SM:=H^{-1}(\frac{1}{2}).\) This is not a genuine restriction from a dynamical point of view, since other energy levels may be understood by simply changing \(\Omega\) to \(\lambda \Omega\), where \(\lambda\in \mathbb{R}.\)

\subsection{\normalsize Connections and Higgs Fields}
\indent \indent Ray transforms may be studied from the perspective of the relevant transport equation. That is, the attenuated ray transform of a given function \(f\in C^{\infty}(M)\), with attenuation \(a\in C^{\infty}(M)\), may be defined as the solution of the boundary value problem:
\[Xu+au=-f\ \text{in}\ SM,\ \ u|_{\partial_{-}(SM)}=0\]
where \(u:SM\rightarrow\mathbb{C}\),\ \(\partial_{\pm}(SM):=\left\{(x,v)\in SM:x\in\partial M,\ \pm\left\langle v,\nu\right\rangle \leq 0\right\}\), and \(\nu\) is the outward unit normal. In the above equation the vector field \(X\) is the generator of the geodesic flow, and is the correct operator to consider if one wishes to define the geodesic ray transform. For the purposes of this paper we will work with the generator of the magnetic flow: \(X(x,v)+Y^{i}_{k}(x)v^{i}\frac{\partial}{\partial v^{k}}\). It is a routine exercise to show that for \(2\)-manifolds this vector field in fact simplifies to: \(X+\lambda V,\) where \(\lambda\in C^{\infty}(M)\) is the unique function satisfying: \(\Omega=\lambda dV_{g}\) for \(dV_{g}\) the area form of \(M.\) Furthermore, throughout this paper we will follow \cite{paternain111} in considering a generalization of the above setup to systems of equations. Therefore, for us the attenuation is given by a unitary connection \(A:TM\rightarrow\mathfrak{u}(n)\) on the trivial bundle \(M\times\mathbb{C}^{n}\). This is a smooth mapping into \(\mathfrak{u}(n)\), the skew-Hermitian complex \(n\times n\) matrices, which for fixed \(x\in M\) is linear in \(v\in T_{x}M\). Note that as usual a connection induces a covariant derivative which acts on sections of \(M\times\mathbb{C}^{n}\) by \(d_{A}:=d+A.\) In what follows the ray transform will also be attenuated by a \textit{Higgs field}, which for our purposes is defined as a smooth map: \(\Phi:M\rightarrow\mathfrak{u}(n).\) Pairs \((A,\Phi)\) often arise in physical contexts such as Yang-Mills theory, see \cite{dunajski10}. Collecting the above ideas, the transport equation to be considered becomes: 
\[(X+\lambda V)u+\mathcal{A}u=-f\ \ \text{in}\ SM,\ \ u|_{\partial_{-}(SM)}=0,\]
where \(u,f:SM\rightarrow \mathbb{C}^{n},\) and \(\mathcal{A}:=A + \Phi.\) On any fixed magnetic geodesic this equation is a linear system of ODEs with zero initial condition, and thus has a unique solution which we will denote by: \(u=u^{f}.\)
\begin{definition} The magnetic ray transform of \(f\in C^{\infty}(SM,\mathbb{C}^{n})\) with attenuation determined by \(\mathcal{A}:=A + \Phi\) is given by
\[I_{\mathcal{A}}f:=I_{A,\Phi}f:=u^{f}|_{\partial_{+}(SM)}.\]\label{defntransform}
\end{definition}

\subsection{\normalsize Summary of the Results}
\indent Assume that \((M,g)\) is a compact, oriented Riemannian \(n\)-manifold with boundary. \(SM\) will denote the unit tangent bundle, a compact \((2n-1)\)-manifold with boundary: \(\partial(SM)=\left\{(x,v)\in SM:\ x\in\partial M\right\}.\) We denote the outer unit normal by \(\nu\), and delineate two important subsets by: \[\partial_{\pm}(SM):=\left\{   (x,v)\in SM:\ x\in\partial M,\ \pm\left\langle v,\nu\right\rangle\leq 0\right\}.\]  We denote by \(\tau(x,v)\geq 0\) the time when the magnetic geodesic determined by \((x,v)\in SM\) exits \(M\). The manifold is said to be \textit{nontrapping} if \(\tau(x,v)\) is always finite. The boundary of \(M\) is said to be \textit{strictly magnetic convex} if \[\Pi(x,v)>\left\langle Y_{x}(v),\nu(x)\right\rangle,\ \ (x,v)\in S(\partial M)\] where \(\Pi\) is the second fundamental form of \(\partial M.\) Fixing \(x \in M\) we define the magnetic exponential map at \(x\) to be the partial map \(\text{exp}_{x}^{\mu}:T_{x}M\rightarrow M\) defined by: \[\text{exp}_{x}^{\mu}(tv):=\pi\circ\varphi_{t}(v),\ \ t\geq 0,\ v\in S_{x}M.\] One can show, see \cite{dairbekov107} that \(\text{exp}_{x}^{\mu}\) is a \(C^{1}\)-partial map on \(T_{x}M\) which is \(C^{\infty}\) on \(T_{x}M\backslash\left\{0\right\}.\) We say that a compact Riemannian manifold \((M,g)\) with a magnetic flow arising from the closed \(2\)-form \(\Omega\) on \(M\) is \textit{simple} if \(\partial M\) is strictly magnetic convex, and the magnetic exponential map \(\text{exp}_{x}^{\mu}:(\text{exp}_{x}^{\mu})^{-1}(M)\rightarrow M\) is a \(C^{1}\)-diffeomorphism \(\forall x\in M.\) (Compare this to the definition of a simple Riemannian manifold, with no magnetic field present in \cite{michel80}, which was motivated by the boundary rigidity problem.) These hypotheses imply that \(M\) is diffeomorphic to the unit ball in \(\mathbb{R}^{n}\), in particular, it is topologically trivial. This allows one to work with a primitive of \(\Omega\), that is a \(1\)-form \(\alpha\) such that \(d\alpha=\Omega.\) We shall say that a triple \((M,g,\alpha)\) is a \textit{simple magnetic system} if it satisfies the above hypotheses. For the rest of the paper we will work with simple magnetic systems. The first main result is the following:

\begin{theorem} Let \((M,g,\alpha)\) be a simple magnetic system with \(\operatorname{dim}(M)=2.\) Assume that \(f:SM\rightarrow\mathbb{C}^{n}\) is a smooth function such that \(f=F+\sigma\), where \(F:M\rightarrow\mathbb{C}^{n}\) is a smooth function, and \(\sigma\) is a \(\mathbb{C}^{n}\)-valued 1-form. Let \(A:SM\rightarrow\mathfrak{u}(n)\) be a unitary connection, and \(\Phi:M\rightarrow\mathfrak{u}(n)\) a skew-Hermitian matrix function. If \(I_{A,\Phi}(f)=0\), then \(F=\Phi p\) and \(\sigma=d_{A}p\), where \(p:M\rightarrow\mathbb{C}^{n}\) is smooth function with \(p|_{\partial(M)}=0.\)\label{mainresult}
\end{theorem}

The above question with no magnetic field present was addressed in \cite{paternain111}, and also in \cite{salo11} with \(n=1\) and no Higgs field.\\
\indent Given any nontrapping compact manifold \((M,g)\) with strictly convex boundary, we define the \textit{scattering relation} \(\mathcal{S}:\partial_{+}(SM)\rightarrow\partial_{-}(SM)\) to map a starting point and direction \((x,v)\in\partial_{+}(SM)\) to its exit point and direction under the magnetic flow. Since we are working with a connection and Higgs field on the trivial bundle we obtain another piece of scattering data as follows. Let \(U_{+}:SM\rightarrow GL(n,\mathbb{C})\) be the unique matrix solution to the transport equation:
\[(X+\lambda V)U_{+}+AU_{+}+\Phi U_{+}=0\ \text{in}\ SM,\ \ U_{+}|_{\partial_{-}(SM)}=\operatorname{Id}.\]
\begin{definition} The scattering relation associated to the attenuation pair \((A,\Phi)\) is given by the map:
\[C^{A,\Phi}_{+}:\partial_{+}(SM)\rightarrow GL(n,\mathbb{C}),\ \ C^{A,\Phi}_{+}:=U_{+}|_{\partial_{+}(SM)}.\]
We will also use the notation \(C^{\mathcal{A}}_{+}\) when \(\mathcal{A}=A+\Phi.\)
\end{definition}
This allows us to pose another inverse question: given a fixed simple magnetic system, can one recover \((A,\Phi)\) from the knowledge of \(C^{A,\Phi}_{+}\)? Observe that the scattering relation exhibits the following gauge invariance:
\[C^{Q^{-1}(X+\lambda V+A)Q,Q^{-1}\Phi Q}_{+}=C^{A,\Phi}_{+},\ \ \text{if}\ Q\in C^{\infty}(M,GL(n,\mathbb{C}))\ \text{satisfies}\ Q|_{\partial M}=\operatorname{Id}.\]
Indeed, suppose \(Q\in C^{\infty}(M,GL(n,\mathbb{C}))\), then \(Q^{-1}U_{+}\) satisfies
\[(X+\lambda V + Q^{-1}(X+\lambda V + A)Q+Q^{-1}\Phi Q)(Q^{-1}U_{+})=0\ \text{in}\ SM,\ \ Q^{-1}U_{+}|_{\partial_{-}(SM)}=Q^{-1}|_{\partial_{-}(SM)}\]
Now we are working with a unitary \(A\), and skew-Hermitian \(\Phi\), therefore \(U_{+}\) and \(C^{A,\Phi}_{+}\) will take values in \(U(n)\), and from the above calculation \(C^{A,\Phi}_{+}\) will be invariant under unitary gauge transformations which are the identity on the boundary. Hence, we have to modify our question to account for this natural obstruction. The following theorem gives a positive answer when \(\text{dim}(M)=2.\)

\begin{theorem} Let \((M,g,\alpha)\) be a simple magnetic system with \(\operatorname{dim}(M)=2.\) Let \(A\) and \(B\) be two unitary connections, and let \(\Phi\) and \(\Psi\) be two skew-Hermitian matrix functions. Then \(C^{A,\Phi}_{+}=C^{B,\Psi}_{+}\) implies that there exists a smooth \(U:M\rightarrow U(n)\) such that \(U|_{\partial{M}}=\operatorname{Id}\) and \(B=U^{-1}dU+U^{-1}AU\), \(\Psi=U^{-1}\Phi U.\)\label{gauge}
\end{theorem}

\indent In the final section of the paper we provide an application to tensor tomography. If \(h\) is a symmetric covariant \((m-1)\)-tensor field we define its inner derivative \(d^{s}h\) to be a symmetric \(m\)-tensor field as follows \(d^{s}h:=\sigma\nabla h\), where \(\sigma\) denotes symmetrization, and \(\nabla\) is the Levi-Civita connection. Now by a well known result, see \cite{sharafutdinov94}, one can uniquely decompose a symmetric \(m\)-tensor \(f\) into \(f=f^{s}+d^{s}h\) where \(h\) is a symmetric \((m-1)\)-tensor which vanishes on the boundary, and \(f^{s}\) is a symmetric \(m\)-tensor with zero divergence. One refers to \(f^{s}\) and \(d^{s}h\) as the solenoidal and potential components of \(f\) respectively.\\
\indent Given a symmetric covariant \(m\)-tensor field \(f=f_{i_{1}\cdots i_{m}}dx^{i_{1}}\otimes\cdots\otimes dx^{i_{m}} \) on \(M\) we have an associated function on \(SM\) defined by \[\hat{f}(x,v):=f_{i_{1}\cdots i_{m}}v^{i_{1}}\cdots v^{i_{m}}.\] Given any continuous function on \(SM\) we may consider its image under the ray transform. The tensor tomography question asks what is the kernel of the ray transform in the space of all symmetric covariant tensors. Historically, the question has been posed for the geodesic ray transform. By a direct calculation one can show for a symmetric \((m-1)\)-tensor \(h\), that \(dh(x,v)=Xh(x,v)\), therefore, by the Fundamental Theorem of Calculus, if \(h|_{\partial M}=0\), then \(dh\) will be contained in the kernel of the geodesic ray transform. In \cite{paternain211} it was shown that when \((M,g)\) is a simple Riemannian \(2\)-manifold these are the only elements in the kernel, and this holds for symmetric tensors of all ranks \(m\geq 0\).\\
\indent In this paper we consider the analogous result in the magnetic setting, thus we fix a simple magnetic system \((M,g,\alpha)\), where \(\operatorname{dim}(M)=2.\) We will denote the unattenuated magnetic ray transform by \(I\). Note that its kernel will have a more convoluted form than its geodesic counterpart because purely potential elements which vanish on the boundary will no longer necessarily map to \(0.\) Nevertheless, using the methods of \cite{paternain211} we generalize the result obtained there to obtain:

\begin{theorem} Let \((M,g,\alpha)\) be a simple magnetic system with \(\operatorname{dim}(M)=2\). Suppose \(f_{i}\) are symmetric i-tensor fields on \(M\) for \(0\leq i\leq k.\) If \(I(\sum_{i=0}^{k}\hat{f}_{i})=0,\) then 
\begin{equation}
(\sum_{i=0}^{k}\hat{f}_{i}) =(X+\lambda V)\left(\sum_{i=0}^{k-1}\hat{u}_{i}\right) \label{onelasttime}
\end{equation}
where  \(u_{i}\) are symmetric i-tensor fields such that \(\hat{u}_{i}|_{\partial(SM)}=0.\) If \(k=0\), then \(f_{0} = 0\). \label{tensortomography}
\end{theorem}

We note that the case where \(k=2\) was considered in \cite{dairbekov107}. There it was shown using pseudodifferential techniques that the magnetic ray transform is injective on symmetric tensors of degree at most \(2\) (up to the natural obstruction described in our theorem above) whenever \((g,\alpha)\) are analytic, and also for generic \((g,\alpha)\). Our theorem thus recovers their result, and moreover, holds for tensors of arbitrary rank.\\
\indent The proof we employ involves considering holomorphic integrating factors for functions and 1-forms simultaneously, since the magnetic field couples the equations.

\subsection{\normalsize Acknowledgements} 
I would like to thank my advisor, Gabriel Paternain, for all his support and encouragement. Thanks also to Nurlan Dairbekov and Yernat Assylbekov for some helpful comments on an earlier draft. Finally I gratefully acknowledge the Fields Institute for their hospitality during their Program on Geometry in Inverse Problems 2012, where this work was partially carried out.

\section{Preliminaries}
\indent \indent Let \(M\) be a compact, oriented Riemannian \(2\)-manifold with boundary. As previously \(X\) denotes the vector field on \(SM\) generated by the geodesic flow. Since \(M\) is oriented, \(SM\) is an \(S^{1}\)-fibration with a circle action on the fibres inducing a vector field which we shall denote by \(V\). By defining \(X_{\bot}:=[X,V]\), we obtain a global frame \(\left\{ X,X_{\bot},V \right\}\) for \(T(SM)\). The two remaining commutators will play an important role in what follows, and are given by (see \cite{kazhdan80}): \([V,X_{\bot}]=X\), and \([X,X_{\bot}]=-KV\), where \(K\) is the Gaussian curvature of \(M\). We define a Riemannian metric on \(SM\) by declaring that \(\left\{X,X_{\bot},V\right\}\) form an orthonormal basis, and will denote by \(d\Sigma^{3}\) the volume form of this metric.\\
\indent We define an inner product between functions \(u,v:SM\rightarrow \mathbb{C}^{n}\) as follows: 
\[\left\langle u,v\right\rangle:=\int_{SM}\left\langle u,v\right\rangle _{\mathbb{C}^{n}} d\Sigma^{3}.\]
Now, \(L^{2}(SM,\mathbb{C}^{n})\) decomposes orthogonally as
\[L^{2}(SM,\mathbb{C}^{n})=\bigoplus_{k\in\mathbb{Z}}H_{k}\]
where \(-iV\) acts as \(k\operatorname{Id}\) on \(H_{k}.\) Thus, we can decompose a smooth function \(u:SM\rightarrow\mathbb{C}^{n}\) into its Fourier components 
\[u=\sum^{\infty}_{k=-\infty}u_{k}\]
where \(u_{k}\in\Omega_{k}:=C^{\infty}(SM,\mathbb{C}^{n})\cap H_{k}.\)\\
\indent We next introduce a fundamental operator \(\mathcal{H}\) on \(L^{2}(SM,\mathbb{C}^{n})\) known as the Hilbert transform. This operator was used systematically in \cite{pestov05}. We will follow the approach of \cite{paternain111} and introduce it via a fibrewise definition. Given \(u\in L^{2}(SM,\mathbb{C}^{n})\) we have \(\mathcal{H}(u):=\sum_{k}\mathcal{H}(u_{k})\) and,
\[\mathcal{H}(u_{k}):=-\operatorname{sgn}(k)iu_{k}\]
where we adopt the convention that \(\operatorname{sgn}(0)=0.\) Observe that
\[(\operatorname{Id}+i\mathcal{H})u=u_{0}+2\sum^{\infty}_{k=1}u_{k},\]
\[(\operatorname{Id}-i\mathcal{H})u=u_{0}+2\sum^{-1}_{k=-\infty}u_{k}.\]
\begin{definition} A function \(u:SM\rightarrow\mathbb{C}^{n}\) is said to be holomorphic if \((\operatorname{Id}-i\mathcal{H})u=u_{0}.\) Equivalently, \(u\) is holomorphic if \(u_{k}=0\) for all \(k<0.\) Similarly, \(u\) is said to be antiholomorphic if \((\operatorname{Id}+i\mathcal{H})u=u_{0}\), or equivalently, if \(u_{k}=0\) for all \(k>0.\)
\end{definition}
We now introduce a commutator formula for the Hilbert transform and the magnetic flow. This formula first appeared in \cite{pestov05} (in a less general formulation), and has since often been used in the resolution of other inverse problems, see \cite{dairbekov107, salo11, paternain111, paternain211}.

\begin{proposition} \([\mathcal{H}, X+\lambda V+A+\Phi]u=(X_{\bot}+\ast A)u_{0}+\left\{(X_{\bot}+\ast A)u\right\}_{0}.\)\label{commutatorformula}
\end{proposition}
\begin{proof} The proof follows immediately from that given in \cite{paternain111} once one notices that \([\mathcal{H},\lambda V]=0.\)
\end{proof}

\section{Holomorphic Integrating Factors}
\indent \indent In Definition~\ref{defntransform} we defined the magnetic ray transform of \(f\in C^{\infty}(SM,\mathbb{C}^{n})\) as \(u^{f}|_{\partial_{+}(SM)}\) where \(u^{f}\) is the solution of the transport equation: 
\[(X+\lambda V)u+\mathcal{A}u=-f\ \ \text{in}\ SM,\ \ u|_{\partial_{-}(SM)}=0.\]
Here we derive an alternative integral representation. Let \(U_{-}:SM\rightarrow GL(n,\mathbb{C})\) be the unique matrix solution to the equation:
\[(X+\lambda V)U_{-}+\mathcal{A}U_{-}=0\ \text{in}\ SM,\ \ U_{-}|_{\partial_{+}(SM)}=\operatorname{Id}.\]
Now \(U^{-1}_{-}\) solves \((X+\lambda V)U^{-1}_{-}-U^{-1}_{-}\mathcal{A}=0.\) Therefore, \((X+\lambda V)(U^{-1}_{-}u^{f})=-U^{-1}_{-}f.\)
Integrating from \(0\) to \(\tau(x,v)\) for \((x,v)\in\partial_{+}(SM)\) we obtain:
\[I_{\mathcal{A}}f(x,v)=\int^{\tau(x,v)}_{0}U^{-1}_{-}(\varphi_{t}(x,v))f(\varphi_{t}(x,v))dt\]
For the rest of this section we will employ the notation \(I_{0}\) and \(I_{1}\) for the unattenuated magnetic ray transform acting on functions (on \(M\)) and \(1\)-forms respectively. Observe that in the unattenuated case when \(n=1\), the above representation yields for \(i=0,1\):
\[I_{i}f(x,v)=\int^{\tau(x,v)}_{0}f(\varphi_{t}(x,v))dt.\]
Here \(\tau(x,v)\) is the exit time of the magnetic geodesic determined by \((x,v)\). We also define \[C^{\infty}_{\alpha}(\partial_{+}(SM)):=\left\{h\in C^{\infty}(\partial_{+}(SM));\  h_{\psi}\in C^{\infty}({SM})\right\}\] where \(h_{\psi}\) is the function defined by \(h\) on \(\partial_{+}(SM)\), and extended to be constant along the orbits of the magnetic flow.

In \cite{dairbekov107} it is shown that the magnetic ray transform extends to a bounded operator from \(L^{2}(SM)\) to \(L^{2}_{\mu}(\partial_{+}(SM)),\) where the latter Hilbert space is defined as the space of real-valued functions \(h\) on \(\partial_{+}(SM)\) for which the following norm is finite:
\[\left\|h\right\|^{2}=\int_{\partial_{+}(SM)}h^{2}d\mu.\]
Here \(d\mu\) is a weighted local product measure, see \cite{dairbekov107} for details. Thus, we are justified in introducing the adjoint of the ray transform. First, we digress to make an important definition.\\

\begin{definition} A holomorphic (respectively antiholomorphic) integrating factor for the equation:
\[Xu+\mathcal{A}u=-f\ \ \text{in}\ SM,\]
(in the case when \(n=1\)) is defined to be any complex function \(\omega\in C^{\infty}(SM)\) which is holomorphic (respectively antiholomorphic), and such that 
\[X\omega=-\mathcal{A}\ \text{in}\ SM.\]
 Here recall that \(\mathcal{A}\) is the sum of a smooth complex function on \(M\) and a complex \(1\)-form. 
\end{definition}

The main ingredient in the proof of Theorem~\ref{mainresult} will turn to be the existence of holomorphic (and antiholomorphic) integrating factors. In turn, the existence of holomorphic integrating factors fundamentally rests on the surjectivity of the adjoint of the unattenuated ray transform on functions and \(1\)-forms. This was proven in the geodesic case in \cite{pestov05} and \cite{pestov04}, and in the magnetic case in \cite{dairbekov107}. Using these results we prove the following two lemmas which will lead swiftly to the desired theorem. We will use the following notation: \(C^{\infty}_{\delta}(M, \Lambda^{1}):=\left\{\beta\in C^{\infty}(M, \Lambda^{1}):\delta\beta=0\right\}.\)

\begin{lemma} The map given by \(S:C^{\infty}_{\alpha}(\partial_{+}(SM))\rightarrow C^{\infty}_{\delta}(M, \Lambda^{1})\), where \(S(h):= X_{\bot}(h_{\psi})_{0}\), is surjective.\label{surj1}
\end{lemma}
\begin{proof} Let \(\beta\in C^{\infty}_{\delta}(M, \Lambda^{1})\). We know that \(d(\ast\beta)=0\) and \(M\) is simply connected, therefore there exists \(F\in C^{\infty}(M)\) such that \(dF=-\ast\beta\). From \cite{dairbekov107} we know that the map \(I^{*}_{0}:C^{\infty}_{\alpha}(\partial_{+}(SM))\rightarrow C^{\infty}(M)\) defined by \(I^{*}_{0}(h):=2\pi(h_{\psi})_{0}\) is surjective. Therefore, there exists \(h\in C^{\infty}_{\alpha}(\partial_{+}(SM))\) such that \(2\pi(h_{\psi})_{0}=F\). Define \(h':=2\pi h\). Now,
\begin{eqnarray*}
S(h')= X_{\bot}(h'_{\psi})_{0} & = & \ast d((h'_{\psi})_{0}) \\
                               & = & \ast d(2\pi(h_{\psi})_{0}) \\
                               & = & \ast dF  \\
                               & = & \beta.
\end{eqnarray*}
Therefore, \(S\) is surjective. Note that to achieve the second equality we used the general fact that \(X_{\bot}f=\ast df\) for any \(f\in C^{\infty}(M)\), which in turn follows from a basic computation using the commutator \(X_{\bot}=[X,V]\).
\end{proof}

\begin{lemma} The map given by \(S':C^{\infty}_{\alpha}(\partial_{+}(SM))\rightarrow C^{\infty}(M)\) where \(S'(h):= (X_{\bot}h_{\psi})_{0}\) is surjective.\label{surj2}
\end{lemma}
\begin{proof} Let \(f\in C^{\infty}(M)\), by Proposition 4.1 in \cite{pestov04} there exists \(v\in C^{\infty}_{\delta}(M,\Lambda^{1})\) such that \(\delta_{\bot}v=f\). To obtain \(\delta_{\bot}v\) one takes the vector field associated to \(v\) by the metric, rotates it by \(90\) degrees using the orientability, and then applies the divergence operator. Now by Theorem 7.3 in \cite{dairbekov107} there exists \(h\in C^{\infty}_{\alpha}(\partial_{+}(SM))\) and harmonic \(h'\in C^{\infty}(M)\) such that \(v=I^{*}_{1}h+\nabla h'.\) Now,
\begin{eqnarray*}
f=\delta_{\bot}v & = & \delta_{\bot}I^{*}_{1}h + \delta_{\bot}\nabla h' \\
                 & = & -\frac{1}{2\pi}\delta_{\bot}I^{*}_{1}(-2\pi h) \\
                 & = & (X_{\bot}((-2\pi h)_{\psi}))_{0}.
\end{eqnarray*}
Therefore \(S'\) is surjective.
\end{proof}

\begin{remark} The last equality in the above proof follows from the general formula: \((X_{\bot}(h_{\psi}))_{0}=-\frac{1}{2\pi}\delta_{\bot}I^{*}_{1}h\) where \(h\in C^{\infty}_{\alpha}(\partial_{+}(SM))\). This formula was implicitly used in the paper \cite{pestov04}, and explicitly written down in \cite{dairbekov107}. It follows by a routine computation.\\
\indent The result of the above lemma in the geodesic case (magnetic field equal to \(0\)) first appeared in \cite{salo11} using a lengthy pseudodifferential argument. We note that emulating the proof of Lemma~\ref{surj2} in the geodesic setting gives a much simpler way of obtaining the result from \cite{salo11}; in this case the surjectivity of \(I^{*}_{1}\) is given by Theorem 4.2 in \cite{pestov04}, and one does not need to invoke any magnetic results from \cite{dairbekov107}.
\end{remark}

\begin{remark} Upon inspection of the proofs of the previous two lemmas, and the statement of Theorem 7.3 in \cite{dairbekov107} which was required therein, one notices that a stronger result holds. Indeed, given \(f=f'+f''\) where \(f\in C^{\infty}(SM),\ f'\in C^{\infty}(M),\ f''\in C^{\infty}_{\delta}(M, \Lambda^{1})\), then there exists \(h\in C^{\infty}_{\alpha}(\partial_{+}(SM))\) such that both \(S'(h)=f'\) and \(S(h)=f''\).
\label{doublesurj}
\end{remark}

\begin{lemma} Let \(f\in C^{\infty}(SM)\) be the sum of a function on \(M\) and a 1-form on \(M\). If \(f\) can be written as \[f=iX_{\bot}((h_{\psi})_{0})+i(X_{\bot}(h_{\psi}))_{0}\] for some \(h\in C^{\infty}_{\alpha}(\partial_{+}(SM))\), then there exists a holomorphic \(\omega\in C^{\infty}(SM)\) such that \((X+\lambda V)\omega=-f\).
\end{lemma}
\begin{proof} Define \(\tilde{\omega}:=h_{\psi}\). Therefore, \(\tilde{\omega}\in C^{\infty}(SM)\) and, \[(X+\lambda V)\tilde{\omega}=0.\] Now define \(\omega:=(\operatorname{Id}+i\mathcal{H})\tilde{\omega}\) so that \(\omega\in C^{\infty}(SM)\) is holomorphic. Now,
\begin{eqnarray*}
(X+\lambda V)\omega & = & (\operatorname{Id}+i\mathcal{H})(X+\lambda V)\tilde{\omega} - i[\mathcal{H},X]\tilde{\omega} \\ 
& = & - iX_{\bot}(\tilde{\omega}_{0}) - i(X_{\bot}\tilde{\omega})_{0}  \\
& = & - iX_{\bot}((h_{\psi})_{0}) - i(X_{\bot}(h_{\psi}))_{0} \\
& = & -f
\end{eqnarray*}
\end{proof}

\begin{theorem} Suppose \(f\in C^{\infty}(SM)\) is given by \(f=f'+f''\) where \(f'\in C^{\infty}(M)\) and \(f''\) is a 1-form on \(M\). Then there exists a holomorphic \(\omega\in C^{\infty}(SM)\) such that \((X+\lambda V)\omega=-f\).\label{thetheorem}
\end{theorem}
\begin{proof} Let us first suppose that \(f''\) is solenoidal. By the above lemma it is sufficient to find \(h\in C^{\infty}_{\alpha}(\partial_{+}(SM))\) such that \[iX_{\bot}((h_{\psi})_{0})+i(X_{\bot}(h_{\psi}))_{0} = f'+f''.\] Using Lemma~\ref{surj1}, Lemma~\ref{surj2}, and Remark~\ref{doublesurj},  allows us to conclude, since the right hand side is the sum of a function and a solenoidal 1-form. \\
\indent For a general \(f''\) decompose it as follows: \(f''=\alpha^{s}+dp\) where \(\alpha^{s}\) is a solenoidal 1-form and \(p\in C^{\infty}(M)\). Now \[ (X+\lambda V)\omega = -f \ \ \Leftrightarrow\ \  (X+\lambda V)\omega' = -f' -\alpha^{s}\] where \(\omega'=\omega+p\). This last expression has a holomorphic solution \(\omega'\in C^{\infty}(SM)\) from above. Therefore \(\omega:=\omega'-p\in C^{\infty}(SM)\) is holomorphic and satisfies \((X+\lambda V)\omega=-f\).    
\end{proof}
Similar proofs yield entirely analogous results in the antiholomorphic case.

\begin{lemma} Let \(f\in C^{\infty}(SM)\) be the sum of a function on \(M\) and a 1-form on \(M\). If \(f\) can be written as \[f=iX_{\bot}((h_{\phi})_{0})+i(X_{\bot}(h_{\phi}))_{0}\] for some \(h\in C^{\infty}_{\alpha}(\partial_{+}(SM))\), then there exists an antiholomorphic \(\omega\in C^{\infty}(SM)\) such that \((X+\lambda V)\omega=-f\).
\end{lemma}

\begin{theorem} Suppose \(f\in C^{\infty}(SM)\) is given by \(f=f'+f''\) where \(f'\in C^{\infty}(M)\) and \(f''\) is a 1-form on \(M\). Then there exists an antiholomorphic \(\omega\in C^{\infty}(SM)\) such that \((X+\lambda V)\omega=-f\).
\end{theorem}

\section{Regularity of the Solution of the Transport Equation}
\indent \indent Recall that given a smooth \(f:SM\rightarrow\mathbb{C}^{n}\) we denote by \(u^{f}\) the unique solution to \[(X+\lambda V)u+\mathcal{A}u =-f,\ \ u|_{\partial_{-}(SM)}=0.\] Note that \(u^{f}\) may fail to be smooth everywhere in general, however, in this section we will closely emulate the results in \cite{paternain111} in order to obtain that \(u^{f}\) is smooth whenever \(f\) is in the kernel of the magnetic ray transform. First we need some preliminaries.\\
\begin{definition} We define the scattering relation \(\mathcal{S}:\partial_{+}(SM)\rightarrow\partial_{-}(SM)\) of a fixed magnetic system \((M,g,\alpha)\) as follows: \[\mathcal{S}(x,v):=(\gamma_{x,v}(\tau(x,v)),\dot{\gamma}_{x,v}(\tau(x,v))).\]
\end{definition}
Here \(\gamma_{x,v}\) is the unique magnetic geodesic determined by \((x,v)\in\partial_{+}(SM).\) Now given \(w\in C^{\infty}(\partial_{+}(SM),\mathbb{C}^{n})\) consider the unique solution \(\overline{w}:SM\rightarrow\mathbb{C}^{n}\) to the transport equation: \[(X+\lambda V)\overline{w}+\mathcal{A}\overline{w}=0,\ \ \overline{w}|_{\partial_{+}(SM)}=w.\]
Observe that \[\overline{w}(x,v)=U_{-}(x,v)w(\mathcal{S}^{-1}\circ\varphi_{\tau(x,v)}(x,v)).\] Now introduce \[Q:C(\partial_{+}(SM),\mathbb{C}^{n})\rightarrow C(\partial(SM),\mathbb{C}^{n})\] by setting 
\[
 (Qw)(x,v) =
  \begin{cases}
   w(x,v) & \text{if } (x,v)\in\partial_{+}(SM) \\
      C^{\mathcal{A}}_{+}(x,v)(w\circ\mathcal{S}^{-1})(x,v)    & \text{if } (x,v)\in\partial_{-}(SM)
  \end{cases}
\]
then
\[\overline{w}|_{\partial(SM)}=Qw.\]
Define \(\mathcal{S}^{\infty}(\partial_{+}(SM),\mathbb{C}^{n}):=\left\{w\in C^{\infty}(\partial_{+}(SM),\mathbb{C}^{n}):\ \ \overline{w}\in C^{\infty}(SM,\mathbb{C}^{n})\right\}.\) We characterize this space as follows:
\begin{lemma} \(\mathcal{S}^{\infty}(\partial_{+}(SM),\mathbb{C}^{n}) =\left\{w\in C^{\infty}(\partial_{+}(SM),\mathbb{C}^{n}):\ \ Qw\in C^{\infty}(\partial(SM),\mathbb{C}^{n})\right\}.\) \label{anotherlemma}
\end{lemma}
\begin{proof} Consider the transport equation where \(\mathcal{A}=0\), \[(X+\lambda V)u=0,\ \ u|_{\partial_{+}(SM)}=s.\] The solution of this will be denoted \[u=s_{\psi}.\] Now define the extension operator \(E:C(\partial_{+}(SM),\mathbb{C}^{n})\rightarrow C(\partial(SM),\mathbb{C}^{n}),\) 
\[
 (Es)(x,v) =
  \begin{cases}
   s(x,v) & \text{if } (x,v)\in\partial_{+}(SM) \\
      s\circ\mathcal{S}^{-1}(x,v) & \text{if } (x,v)\in\partial_{-}(SM)
  \end{cases}
\]
then \[s_{\psi}|_{\partial(SM)}=Es.\] Lemma 7.6 in \cite{dairbekov107} states that \(s_{\psi}\) is smooth iff \(Es\) is smooth. Now we use the methodology described in the proof of Theorem 7.3 in \cite{dairbekov107}. Embed \(M\) into a closed manifold \(\tilde{M}\), and smoothly extend \(g\) and \(\alpha\) to a Riemannian metric and 1-form on \(\tilde{M}\) respectively, and extend \(\mathcal{A}\) to the unit sphere bundle of \(\tilde{M}\). We shall continue to denote the extensions \((g,\alpha, \mathcal{A}).\) Given \(U\subset\tilde{M}\) an open neighbourhood of \(M\) with smooth boundary, then \((\overline{U},g,\alpha)\) will also be a simple magnetic system provided that \(\partial U\) is sufficiently close to \(\partial M.\) We will assume such a \(U\) is fixed. Now consider the unique solution to the transport equation in \(S\overline{U}\):
\[(X+\lambda V)R+\mathcal{A}R=0,\ \ R|_{\partial_{+}(S\overline{U})}=\operatorname{Id}.\]
If we restrict \(R\) to \(SM\) we obtain a smooth map \(R:SM\rightarrow GL(n,\mathbb{C})\). We will also denote the restriction by \(R.\) Define \(p:=R^{-1}|_{\partial_{+}(SM)}\); note that by this we mean the restriction of \(R\) composed with matrix inversion. Then, \[\overline{w}=R(pw)_{\psi}.\] Also, 
\[
 (Qw)(x,v) =
  \begin{cases}
   R(x,v)p(x,v)w(x,v) & \text{if } (x,v)\in\partial_{+}(SM) \\
     R(x,v)((pw)\circ\mathcal{S}^{-1})(x,v) & \text{if } (x,v)\in\partial_{-}(SM)
  \end{cases}
\]
Now, if \(\overline{w}\) is smooth, then so is \(Qw=\overline{w}|_{\partial(SM)}\). Conversely, assume that \(Qw\) is smooth. Since, \(R\) is smooth, then \(Epw=R^{-1}Qw\) is also smooth. Thus, \((pw)_{\psi}\) is smooth, and so the smoothness of \(R\) guarantees the smoothness of \(\overline{w}.\)
\end{proof}

\begin{proposition} Let \(f:SM\rightarrow\mathbb{C}^{n}\) be a smooth function such that \(I_{\mathcal{A}}(f)=0\). Then \(u^{f}:SM\rightarrow\mathbb{C}^{n}\) is smooth. \label{theproposition}
\end{proposition}
\begin{proof} We repeat the procedure described in the above lemma: embed \(M\) into a closed manifold \(\tilde{M}\), and smoothly extend \(g\) and \(\alpha\) to a Riemannian metric and 1-form on \(\tilde{M}\) respectively, and extend \(f\) and \(\mathcal{A}\) to the unit sphere bundle of \(\tilde{M}\). We shall continue to denote the extensions \((g,\alpha, f, \mathcal{A}).\) Given \(U\subset\tilde{M}\) an open neighbourhood of \(M\) with smooth boundary, then \((\overline{U},g,\alpha)\) will also be a simple magnetic system provided that \(\partial U\) is sufficiently close to \(\partial M.\)\\
\indent Now consider the unique solution to the following system in \(S\overline{U}\):  
\[(X+\lambda V)r+\mathcal{A}r=-f,\ \ r|_{\partial_{-}(S\overline{U})}=0.\]
Then, the restriction of \(r\) to \(SM\), still denoted by \(r\), is a smooth solution \(r:SM\rightarrow\mathbb{C}^{n}\) to \((X+\lambda V)r+\mathcal{A}r=-f.\) Note that \(r-u^{f}\) solves \((X+\lambda V+\mathcal{A})(r-u^{f})=0\). Therefore, defining \(w:=(r-u^{f})|_{\partial_{+}(SM)}\), ensures that \(\overline{w}=r-u^{f}\). Hence, \(u^{f}\) is smooth iff \(\overline{w}\) is smooth. Finally, since \(u^{f}|_{\partial_{+}(SM)}=I_{\mathcal{A}}(f)=0\), we know that \(Qw=\overline{w}|_{\partial(SM)}=r|_{\partial(SM)}\) which is smooth, and so by Lemma~\ref{anotherlemma} \(u^{f}\) is also smooth.
\end{proof}

\section{Injectivity of \(I_{\mathcal{A}}\)}
\indent \indent In this section we will prove Theorem~\ref{mainresult}. First we must establish certain identities. The proof of the next two lemmas is essentially identical to that given in \cite{paternain111} and will be omitted.
\begin{lemma}\label{thelemma} For any pair of smooth functions \( u,g: SM \rightarrow\mathbb{C}^{n} \) we have \[\left\langle V(u),g\right\rangle=-\left\langle u,V(g)\right\rangle.\] If, in addition, \(u |_{\partial(SM)}=0,\) then \[\left\langle Pu,g \right\rangle = -\left\langle u,Pg\right\rangle\] where \(P=X+\lambda V+ A+\Phi\) or \(X_{\bot}+\ast A.\) 
\end{lemma}

\begin{lemma} Assume \((X+\lambda V+A+\Phi)(u)=F+ \sigma\) where \(F:M\rightarrow\mathbb{C}^{n}\) is a smooth function and \(\sigma\) is a \(\mathbb{C}^{n}\)-valued 1-form. Then, \[|V[(X+\lambda V + A + \Phi)(u)]|^{2}-|(X+\lambda V + A+ \Phi)(u)|^{2}=-|F|^{2}\leq 0.\]\label{firstlemma}
\end{lemma}

\begin{lemma} Let \((M,g,\alpha)\) be a simple magnetic system with \(\operatorname{dim}(M)=2.\) There exist smooth solutions \(u:SM\rightarrow\mathbb{R}\) to the magnetic Riccati equation: \[(X+\lambda V)(u)+u^{2}+K+X_{\bot}(\lambda)+\lambda^{2}=0.\]
\end{lemma}
\begin{proof} Let us consider first a single magnetic geodesic \(\gamma_{x,v}\) with \((x,v)\in\partial_{+}(SM)\), and assume its exit time is \(\tau(x,v)\), and thus it is defined on \([0,\tau(x,v)]\). Restricted to this geodesic we seek to solve the equation: \(\dot{u}+u^{2}+K+X_{\bot}(\lambda)+\lambda^{2}=0.\) Now, following the notation of \cite{burns02} any magnetic Jacobi field with initial conditions tangent to \(SM\) may be written as \(J(t):=x(t)\gamma (t)+y(t)i\gamma (t)\) where \(x(t)\) and \(y(t)\) satisfy:
\begin{equation} \dot{x}=\lambda(\gamma)y \label{equationone}
\end{equation}
\begin{equation} \ddot{y}+[K(\gamma)-\left\langle \nabla \lambda(\gamma),i\dot{\gamma}\right\rangle+\lambda ^{2}(\gamma)]y=0 \label{jacobiequation}
\end{equation}
This is relevant to our situation because should we find a nowhere vanishing solution \(z(t)\) to (\ref{jacobiequation}), then \(u(t):=\frac{\dot{z}(t)}{z(t)}\) will solve our original equation. The theory of ODEs guarantees that we can solve (\ref{equationone}) and (\ref{jacobiequation}) once we specify the initial conditions. Therefore, obtain solutions with \(\dot{x}(0)=\dot{y}(0)=1\) and \(x(0)=y(0)=0\). Now, from \cite{herreros10} we know that the fact that our manifold has no conjugate points implies that any Jacobi field along a geodesic which vanishes initially, and is parallel to that geodesic at a later point in time, must be identically zero. Hence, \(y(t)\) since vanishing initially, can never hit \(0\) again. Now we must do something to take care of its initial zero. To this end choose \(w\) a solution of (\ref{jacobiequation}) such that \(\dot{w}(0)=0\) and \(w(0)=1\). Now by continuity \(w(t)\) must be positive on \([0,t_{w}]\) for some \(t_{w}>0.\) Note that \(y(t)\) is bounded from below on \([t_{w},\tau(x,v)]\). Therefore, \(z(t):=cy(t)+w(t)\) is a solution of (\ref{jacobiequation}) where \(c>0\) is chosen sufficiently large such that \(z(t)\neq 0\ \ \ \forall t\in [0,\tau(x,v)].\) Furthermore, by considering \(z(t)=z(x,v,t)\) as depending on the additional parameters \(x\) and \(v\), specifying the geodesic, we see that the Theory of ODEs will further guarantee that \(z(x,v,t)\) depends smoothly on all of its parameters (here we must use compactness to choose \(c\) such that our construction holds for all magnetic geodesics).\\
\indent Set \(G:=\left\{(x,v,t):(x,v)\in\partial_{+}(SM),\ \ 0\leq t\leq\tau(x,v)\right\}.\) Now we define a map \(\pi:SM\rightarrow G\) as follows: \(\pi(x,v):=(x',v',t)\) where \((x',v')\in\partial_{+}(SM)\) determines the unique magnetic geodesic which passes through \((x,v)\) at time \(t\). Thanks to the discussion in Section \(2.2\) of \cite{dairbekov107} we know that this map will be smooth on \(SM\backslash S(\partial M).\) In order to obtain smoothness everywhere we employ a technique used earlier: embed \(M\) into a closed manifold \(\tilde{M}\), and smoothly extend \(g\) and \(\alpha\) to a Riemannian metric and 1-form on \(\tilde{M}\) respectively. We shall continue to denote the extensions \((g,\alpha).\) Given \(U\subset\tilde{M}\) an open neighbourhood of \(M\) with smooth boundary, then \((\overline{U},g,\alpha)\) will also be a simple magnetic system provided that \(\partial U\) is sufficiently close to \(\partial M.\) We now repeat the entire argument above for \(\overline{U}\), and obtain that the analogous maps \(\pi:S\overline{U}\rightarrow G'\) and \(z: G'\rightarrow \mathbb{R}\). Now define \(u(x,v,t):= \frac{\dot{z}(x,v,t)}{z(x,v,t)}\), and by setting \(\tilde{u}:=u\circ\pi|_{SM}\), we obtain a smooth solution to the magnetic Riccati equation on \(SM\).
\end{proof}

\begin{lemma} Let \((M,g,\alpha)\) be a simple magnetic system with \(\operatorname{dim}(M)=2.\) If \(u:SM\rightarrow \mathbb{C}^{n}\) is a smooth function such that \(u|_{\partial(SM)}=0\), then \[|(X+\lambda V + A +\Phi)(Vu)|^{2}- \left\langle (K+X_{\bot}(\lambda)+\lambda^{2}) V(u),V(u)\right\rangle \geq 0.\]\label{secondlemma}
\end{lemma}
\begin{proof} Consider a smooth function \(r:SM\rightarrow\mathbb{R}\) which solves the Riccati-type equation: \((X+\lambda V)(r)+r^{2}+K+X_{\bot}(\lambda)+\lambda^{2}=0.\) These exist by the above result. Now we compute using the fact that \(A+\Phi\) is skew-Hermitian:
\begin{align*}
\|& (X+\lambda V+A+\Phi)(V(u))-rV(u) \| ^{2}_{\mathbb{C}^{n}} \\
                                 & = \left\|(X+\lambda V+A+\Phi)(V(u))\right\|^{2}_{\mathbb{C}^{n}} -2\operatorname{Re}\left\{  \left\langle (X+\lambda V+A+\Phi)(V(u)),rV(u)\right\rangle_{\mathbb{C}^{n}} \right\}  +   r^{2}\left\|V(u)\right\|^{2}_{\mathbb{C}^{n}}\\
                                 & = \left\|(X+\lambda V+A+\Phi)(V(u))\right\|^{2}_{\mathbb{C}^{n}} -2r\operatorname{Re}\left\{ \left\langle (X+\lambda V)(V(u)),V(u)\right\rangle_{\mathbb{C}^{n}} \right\} + r^{2}\left\|V(u)\right\|^{2}_{\mathbb{C}^{n}}
\end{align*}
\indent Employing the Riccati-type equation we deduce: \[(X+\lambda V)(r\left\|V(u)\right\|^{2}_{\mathbb{C}^{n}}) = (-r^{2}-K-X_{\bot}(\lambda)-\lambda^{2})\left\|V(u)\right\|^{2}_{\mathbb{C}^{n}} + 2r\operatorname{Re}\left\{\left\langle (X+\lambda V)(V(u)),V(u)\right\rangle_{\mathbb{C}^{n}} \right\}\]
Therefore, 
\begin{align*}
\|& (X+\lambda V+A+\Phi)(V(u))-rV(u) \| ^{2}_{\mathbb{C}^{n}} \\
     & = \left\|(X+\lambda V+A+\Phi)(V(u))\right\|^{2}_{\mathbb{C}^{n}} +(-K-X_{\bot}(\lambda)-\lambda^{2})\left\|V(u)\right\|^{2}_{\mathbb{C}^{n}}-(X+\lambda V)(r\left\|V(u)\right\|^{2}_{\mathbb{C}^{n}})
\end{align*}
Now integrate this expression over \(SM\) with respect to \(d\Sigma^{3}\) and use that \(V(u)\) vanishes on \(\partial(SM)\) to obtain: 
\begin{align*}
\|(X+\lambda V & +A+\Phi)(V(u))\|^{2} - \left\langle (K+X_{\bot}(\lambda)+\lambda^{2})V(u),V(u)\right\rangle \\
& = \| (X+\lambda V+A+\Phi)(V(u))-rV(u) \| ^{2}\geq 0
\end{align*}
\end{proof}

\begin{theorem} [\sc{Energy identity}] If \(u:SM\rightarrow \mathbb{C}^{n}\) is a smooth function such that \(u|_{\partial(SM)}=0\), then 
\begin{align*}
|(X+\lambda &V+A+\Phi)V(u)|^{2} - \left\langle \ast F_{A}u,V(u)\right\rangle - \operatorname{Re}\left\{\left\langle (\ast d_{A}\Phi)u, V(u)\right\rangle\right\}  - 2\operatorname{Re}  \left\{\left\langle \lambda V(u), \Phi(u)\right\rangle\right\}\\
&-\operatorname{Re}\left\{\left\langle \Phi u, (X+A+\Phi)(u)\right\rangle\right\}  - \left\langle (K+X_{\bot}(\lambda)+\lambda^{2}) V(u),V(u)\right\rangle \\ 
&= |V[(X+\lambda V+A+\Phi)u]|^{2}  -  |(X+\lambda V +A+\Phi)(u)|^{2}
\end{align*}\label{energy}
\end{theorem}
\begin{proof} We begin by noting the structural equations:
\begin{align*}
[V,X] & =-X_{\bot}\\
[V,X_{\bot}] & =X\\
[X,X_{\bot}] & =-KV
\end{align*}

We also recall from \cite{paternain111} the following commutators:
\[[V,X+A+\Phi]=-X_{\bot}-\ast A\]
\[[V,X_{\bot}+\ast A]=X+A\]
\[[X+A+\Phi,X_{\bot}+\ast A]=-KV-\ast F_{A}-\ast d_{A}\Phi\]
To proceed one needs to include the magnetic term, which after a short computation yields:
\[[V,X+\lambda V +A+\Phi]=-X_{\bot}-\ast A\]
\[[V,X_{\bot}+\ast A]=X+A\]
\[[X+\lambda V+A+\Phi,X_{\bot}+\ast A]=-KV+\lambda X-X_{\bot}(\lambda)V+\lambda A-\ast F_{A}-\ast d_{A}\Phi\]

Now consider the operator \(P:=V(X+\lambda V+A+\Phi)\) acting on \(C^{\infty}(SM,\mathbb{C}^{n})\). We may decompose \(P\) into self-adjoint and skew-adjoint components (with respect to the \(L^{2}\) inner product) as follows: \[P=C+iD,\ \ C:=\frac{1}{2}(P+P^{*}),\ \ D:=\frac{1}{2i}(P-P^{*}).\] 
Note that the formal adjoint of \(P\) is \(P^{*}=(X+\lambda V+A+\Phi)V\). Now, given \(u\in C^{\infty}(SM,\mathbb{C}^{n})\) such that \(u|_{\partial(SM)}=0,\) then it easily follows that \[\left\|Pu\right\|^{2}=\left\|Cu\right\|^{2}+\left\|Du\right\|^{2}+\left\langle i[C,D]u,u\right\rangle.\] Furthermore, one computes \(i[C,D]=\frac{1}{2}[P^{*},P]\) and \(\left\|Cu\right\|^{2}+\left\|Du\right\|^{2}=\frac{1}{2}(\left\|Pu\right\|^{2}+\left\|P^{*}u\right\|^{2}).\) Thus, we may write our expression above as 
\begin{align}
\left\|Pu\right\|^{2}=\left\|P^{*}u\right\|^{2}+\left\langle [P^{*},P]u,u\right\rangle \label{better}
\end{align}
We compute the commutator as follows:
\begin{align*}
\left\langle [P^{*},P]u,u\right\rangle & =  \left\langle (X+\lambda V+A+\Phi)VV(X+\lambda V+A+\Phi)u-V(X+\lambda V+A+\Phi)(X+\lambda V+A+\Phi)Vu,u\right\rangle\\
                    & =  \left\langle V(X+\lambda V+A+\Phi)V(X+\lambda V+A+\Phi)u+(X_{\bot}+\ast A)V(X+\lambda V+A+\Phi)u,u\right\rangle\\
                    & \ \ \ \ \ \ \ \ \ \ -\left\langle V(X+\lambda V+A+\Phi)V(X+\lambda V+A+\Phi)u-V(X+\lambda V+A+\Phi)(X_{\bot}+\ast A)u,u \right\rangle\\
                    & = \left\langle V(X_{\bot}+\ast A)(X+\lambda V+A+\Phi)u-(X+A)(X+\lambda V+A+\Phi)u,u \right\rangle\\
                    & \ \ \ \ \ \ \ \ \ \ -\left\langle V(X+\lambda V+A+\Phi)(X_{\bot}+\ast A)u,u\right\rangle\\
                    & = \left\langle V[X_{\bot}+\ast A, X+\lambda V+A+\Phi]u-(X+A)(X+\lambda V+A+\Phi)u,u\right\rangle\\
                    & = \left\langle V(KV-\lambda X+X_{\bot}(\lambda)V-\lambda A+\ast F_{A}+\ast d_{A}\Phi)u - (X+A)(X+\lambda V+A+\Phi)u,u\right\rangle\\
                    & = -\left\langle(KV-\lambda X + X_{\bot}(\lambda)V-\lambda A+\ast F_{A}+\ast d_{A}\Phi )u,V(u)\right\rangle\\
                    & \ \ \ \ \ \ \ \ \ \ +\left\langle(\lambda V+\Phi)(X+\lambda V+A+\Phi)u,u\right\rangle+|(X+\lambda V+A+\Phi)(u)|^{2}\\
                    & = -\left\langle(K+ X_{\bot}(\lambda)+\lambda^{2})V(u),V(u)\right\rangle-\left\langle\ast F_{A}u,V(u)\right\rangle-\left\langle(\ast d_{A}\Phi)u,V(u)\right\rangle-\left\langle(-\lambda A-\lambda X )u,V(u)\right\rangle\\
                    & \ \ \ \ \ \ \ \ \ \ +\left\langle(\lambda V+\Phi)(X+\lambda V+A+\Phi)u,u\right\rangle+|(X+\lambda V+A+\Phi)(u)|^{2}\\
                    & = -\left\langle(K+X_{\bot}(\lambda)+\lambda^{2})V(u),V(u)\right\rangle-\left\langle\ast F_{A}u,V(u)\right\rangle-\left\langle(\ast d_{A}\Phi)u,V(u)\right\rangle+\left\langle\lambda Au+\lambda Xu,V(u)\right\rangle\\
                    & \ \ \ \ \ \ \ \ \ \ +\left\langle(\lambda V(X+A+\Phi)u,u\right\rangle+ \left\langle \Phi(X+\lambda V+A+\Phi)(u),u\right\rangle+|(X+\lambda V+A+\Phi)(u)|^{2}\\
                    & = -\left\langle(K+X_{\bot}(\lambda)+\lambda^{2})V(u),V(u)\right\rangle-\left\langle\ast F_{A}u,V(u)\right\rangle-\left\langle(\ast d_{A}\Phi)u,V(u)\right\rangle\\
                    & \ \ \ \ \ \ \ \ \ \ +2\operatorname{Re}\left\{\left\langle(\lambda V(u),\Phi u\right\rangle\right\}-\left\langle (X+A+\Phi)(u),\Phi u\right\rangle+|(X+\lambda V+A+\Phi)(u)|^{2}\\
\end{align*}

Substituting the above into (\ref{better}) and taking imaginary parts shows:
\[\left\langle(\ast d_{A}\Phi)u,V(u)\right\rangle+\left\langle (X+A+\Phi)(u),\Phi u\right\rangle = \operatorname {Re}\left\{  \left\langle(\ast d_{A}\Phi)u,V(u)\right\rangle+\left\langle (X+A+\Phi)(u),\Phi u\right\rangle   \right\}\]
Therefore, using our expression for the commutator \(\left\langle [P^{*},P]u,u\right\rangle\) in (\ref{better}) yields Theorem~\ref{energy}.
\end{proof}

\begin{proof}[Proof of Theorem~\ref{mainresult}] \(I_{A,\Phi}=0\) implies that there exists a function \(u:SM\rightarrow\mathbb{C}^{n}\) with \(u|_{\partial(SM)}=0\) which satisfies \[(X+\lambda V+A+\Phi)(u)=-f.\]
Proposition~\ref{theproposition} implies that \(u\in C^{\infty}(SM,\mathbb{C}^{n})\). We seek to show that \(u\) is both holomorphic and antiholomorphic. Given this, then \(u=u_{0}\) only depends on \(x\) and \(u|_{\partial(M)}=0\). Furthermore, we have \[du+Au=-\sigma, \ \ \ \Phi u =-F\] Now our theorem follows upon choosing \(p=-u.\) Thus, it remains to show \(u\) is both holomorphic and antiholomorphic. We will simply explain the holomorphic case (the antiholomorphic case follows by an analogous argument). \\
\indent Choose a real-valued 1-form \(\varphi\) such that \(d\varphi=dV_{g}\) where \(dV_{g}\) is the area form of \(M\). Define \(A_{s}:=A+is\varphi \operatorname{Id}\). Now, \(s\in\mathbb{R}\), \(A_{s}\) is a unitary connection and \(i\ast F_{A_{s}}=i\ast F_{A}-s \operatorname{Id}\). By Theorem~\ref{thetheorem} we can find a holomorphic function \(\omega\in C^{\infty}(SM)\) such that \((X+\lambda V)\omega=-i\varphi\). Thus, \(u_{s}:=e^{s\omega}u\) satisfies \[(X+\lambda V +A_{s}+\Phi)u_{s}=-e^{s\omega}f.\]
Now let \(v:=(\operatorname{Id}-i\mathcal{H})u_{s}-(u_{s})_{0}\). We seek to show that \(v=0\). Using Proposition~\ref{commutatorformula} we have
\begin{align*}
(X+\lambda V +A_{s}+\Phi)v&=(\operatorname{Id}-i\mathcal{H})(X+\lambda V +A_{s}+\Phi)u_{s}+i[\mathcal{H},X+A_{s}]u_{s}\\
&\ \ \ \ \ \ \ \ \ \ \ \ \ \ \ \ \ \ \ \  -(X+\lambda V +A_{s}+\Phi)(u_{s})_{0}\\
& = -(\operatorname{Id}-i\mathcal{H})e^{s\omega}f+i(X_{\bot}+\ast A_{s})(u_{s})_{0}+i\left\{(X_{\bot}+\ast A_{s})u_{s}\right\}_{0}\\
&\ \ \ \ \ \ \ \ \ \ \ \ \ \ \ \ \ \ \ \  -(X+\lambda V +A_{s}+\Phi)(u_{s})_{0}.
\end{align*}
Thus, \[(X+\lambda V +A_{s}+\Phi)v=h+\beta\]
where \(h\in C^{\infty}(M,\mathbb{C}^{n})\) and \(\beta\) is a 1-form. Now apply the energy identity with the connection \(A_{s}\) and the function \(v\) (which also satisfies \(v|_{\partial(SM)}=0)\) to deduce:
\begin{align}
|(X+\lambda &V+A_{s}+\Phi)V(u)|^{2} - \left\langle \ast F_{A_{s}}u,V(u)\right\rangle - \operatorname{Re}\left\{\left\langle (\ast d_{A_{s}}\Phi)u, V(u)\right\rangle\right\}  - 2\operatorname{Re}  \left\{\left\langle \lambda V(u), \Phi(u)\right\rangle\right\} \nonumber \\
&-\operatorname{Re}\left\{\left\langle \Phi u, (X+A_{s}+\Phi)(u)\right\rangle\right\}  - \left\langle (K+X_{\bot}(\lambda)+\lambda^{2}) V(u),V(u)\right\rangle \nonumber \\ 
&- |V[(X+\lambda V+A_{s}+\Phi)u]|^{2}  +  |(X+\lambda V +A_{s}+\Phi)(u)|^{2}=0.\label{energyidentity}
\end{align}
Now we invoke Lemma~\ref{firstlemma} and Lemma~\ref{secondlemma} to obtain:
\begin{align}
|(X+\lambda V + A_{s} +\Phi)(V(v))|^{2}- \left\langle (K+X_{\bot}(\lambda)+\lambda^{2}) V(v),V(v)\right\rangle \geq 0 \label{1}
\end{align}
\begin{align}
|(X+\lambda V + A_{s}+ \Phi)(v)|^{2}-|V[(X+\lambda V + A_{s} + \Phi)(v)]|^{2}=|h|^{2}\geq 0. \label{2}
\end{align}
From \cite{paternain111} we know we have the following estimates:
\begin{align} 
-\left\langle   \ast F_{A_{s}}v, V(v)\right\rangle= \sum^{-1}_{k=-\infty}{|k|\left\langle -i\ast F_{A_{s}}v_{k},v_{k}\right\rangle}\geq(s-\left\|F_{A}\right\|_{L^{\infty}(M)})\sum^{-1}_{k=-\infty}{|k||v_{k}|^{2}} \label{3}
\end{align}
\begin{align} \operatorname{Re}\left\{\left\langle (\ast d_{A}\Phi)v,V(v) \right\rangle\right\} \leq C_{A,\Phi}\sum^{-1}_{k=-\infty}{|k||v_{k}|^{2}} \label{4}
\end{align}
for some \(C_{A,\Phi}\in\mathbb{R}\).
\begin{align}
\operatorname{Re}\left\{\left\langle \Phi v,(X+ A_{s}+\Phi)v\right\rangle\right\}  \leq C_{A,\Phi}\sum^{-1}_{k=-\infty}{|v_{k}|^{2}} \label{5}
\end{align}
for some \(C_{A,\Phi}\in\mathbb{R}\).
Lastly, it is easy to see that we have the bound:
\begin{align} 
-\operatorname{Re}\left\{  \left\langle \lambda V(v),\Phi v\right\rangle \right\}\geq -C_{\Phi}\sum^{-1}_{k=-\infty}{|k||v_{k}|^{2}} \label{6}
\end{align}
for some \(C_{\Phi}\in\mathbb{R}\).
Using the estimates (\ref{1})-(\ref{6}) in (\ref{energyidentity}) we obtain: \[0\geq |h|^{2}+(s-C_{A,\Phi})\sum^{-1}_{k=-\infty}{|k||v_{k}|^{2}}\] for some \(C_{A,\Phi}\in\mathbb{R}\). Finally, choosing \(s\) sufficiently large ensures that \(v_{k}=0\) for all \(k\). Thus, \(u_{s}\) is holomorphic, and so \(u=e^{-s\omega}u_{s}\) is also holomorphic.
\end{proof}

\section{Scattering Data Determines the Connection and Higgs Field}
\indent Recall we define \(C^{A,\Phi}_{+}:\partial_{+}(SM)\rightarrow GL(n,\mathbb{C})\) by \[C^{A,\Phi}_{+}:=U_{+}|_{\partial_{+}(SM)},\] where \(U_{+}:SM\rightarrow GL(n,\mathbb{C})\) is the unique matrix solution to the transport equation: \[(X+\lambda V)U_{+}+\mathcal{A}U_{+}=0,\ \ U_{+}|_{\partial_{-}(SM)}=\operatorname{Id}.\]

\begin{proof}[Proof of Theorem~\ref{gauge}] Consider the unique matrix solution \(U_{A,\Phi}:SM\rightarrow U(n)\) to the transport equation: \[(X+\lambda V+A+\Phi)U_{A,\Phi}=0,\ \ U_{A,\Phi}|_{\partial_{-}(SM)}=\operatorname{Id}.\] Then \(C^{A,\Phi}_{+}=U_{A,\Phi}|_{\partial_{+}(SM)}.\) Also consider the unique matrix solution \(U_{B,\Psi}:SM\rightarrow U(n)\) to the transport equation: \[(X+\lambda V+B+\Psi)U_{B,\Psi}=0,\ \ U_{B,\Psi}|_{\partial_{-}(SM)}=\operatorname{Id}.\] Then \(C^{B,\Psi}_{+}=U_{B,\Psi}|_{\partial_{+}(SM)}.\) Now define \(U:=U_{A,\Phi}(U_{B,\Psi})^{-1}:SM\rightarrow U(n).\) This \(U\) satisfies the following system: \[(X+\lambda V)U+AU+\Phi U-UB-U\Psi=0,\ \ U|_{\partial (SM)}=\operatorname{Id}.\] We seek to show that \(U\) is smooth and dependent only on the basepoint \(x\in M\), thereby being our sought function: \(U:M\rightarrow U(n).\)\\
\indent Define \(W:=U-\operatorname{Id}:SM\rightarrow \mathbb{C}^{n\times n},\) where \(\mathbb{C}^{n\times n}\) is to be thought of as the set of all \(n \times n\) complex matrices. Now, \(W\) satisfies:
\begin{equation} (X+\lambda V)W+AW-WB+\Phi W-W\Psi=B-A+\Psi-\Phi, \label{no1}
\end{equation}
\begin{equation}W|_{\partial(SM)}=0. \label{no2}
\end{equation}
\indent Introduce a new connection \(\hat{A}\) on the bundle \(M\times \mathbb{C}^{n\times n}\), and a new Higgs field \(\hat{\Phi}\) as follows: given a matrix \(R\in\mathbb{C}^{n\times n}\) we define \(\hat{A}(R):=AR-RB\) and \(\Phi(R):=\Phi R-R\Psi.\) One can check that \(\hat{A}\) is a unitary connection if \(A\) and \(B\) are, and that \(\hat{\Phi}\) is skew-Hermitian when \(\Phi\) and \(\Psi\) are. Then (\ref{no1}) and (\ref{no2}) can be interpreted as saying that \(I_{\hat{A},\hat{\Phi}}(A-B+\Phi-\Psi)=0.\) Note that any solution \(W:SM\rightarrow \mathbb{C}^{n\times n}\) to (\ref{no1}) with \(W|_{\partial_{-}(SM)}=0\) is unique. Now we invoke Theorem~\ref{mainresult} to conclude that there exists a smooth function \(p:M\rightarrow\mathbb{C}^{n\times n}\) such that \(p\) vanishes on \(\partial M,\) and \(B-A=d_{\hat{A}}p=dp+Ap-pB\) and \(\Psi-\Phi=\hat{\Phi}p=\Phi p-p\Psi.\) Therefore, \[(X+\lambda V)p+Ap-pB+\Phi p-p\Psi=B-A+\Psi-\Phi.\] But now, by uniqueness, \(W=p\), and thus \(U:=W+\operatorname{Id}\) is the required matrix function.
\end{proof}

\section{An Application to Tensor Tomography}
\indent We begin with a proposition which generalizes a result from \cite{salo11}. In this section we will say that \(f\in C^{\infty}(SM)\) has degree \(m\) if \(f_{k}=0\) for \(|k|\geq m+1\), and recall that we denote the unattenuated magnetic ray transform by \(I\).
\begin{proposition} Let \((M,g,\alpha)\) be a simple magnetic system with \(\operatorname{dim}(M)=2.\) Let \(f\) be a smooth holomorphic function on \(SM\). Suppose \(u\in C^{\infty}(SM)\) satisfies \[(X+\lambda V)u=-f,\ \ \ u|_{\partial(SM)}=0.\] Then \(u\) is holomorphic and \(u_{0}=0.\)\label{solutionhol}
\end{proposition}
\begin{proof} It is sufficient to show that \((\operatorname{Id}-i\mathcal{H})u=0\). Using Proposition~\ref{commutatorformula} we have, 
\begin{eqnarray*}
(X+\lambda V)(\operatorname{Id}-i\mathcal{H})u & = & (\operatorname{Id}-i\mathcal{H})(X+\lambda V)u+[\mathcal{H},X]u \\
                                               & = & -(\operatorname{Id}-i\mathcal{H})f + iX_{\bot}(u_{0})+i(X_{\bot}(u))_{0}.
\end{eqnarray*}
Therefore, \[(X+\lambda V)(\operatorname{Id}-i\mathcal{H})u = h + \sigma,\] where \(h=-f_{0}+i(X_{\bot}(u))_{0}\in C^{\infty}(M)\) and \(\sigma=i\ast d(u_{0})\) is a \(1\)-form on \(M\). Now, since \((\operatorname{Id}-i\mathcal{H})u|_{\partial(SM)}=0\) we know that the magnetic ray transform of \((X+\lambda V)(\operatorname{Id}-i\mathcal{H})u\) is \(0.\) Hence, by Theorem~\ref{mainresult} we deduce that: \(h=0\), and \(\sigma=dp\) where \(p\in C^{\infty}(M)\) is such that \(p|_{\partial(M)}=0.\) Now, from \cite{sharafutdinov94} we know that we can uniquely decompose a symmetric tensor into solenoidal and potential components, thus since \(\sigma=i\ast d(u_{0}) = \ast d(iu_{0})\) and \(\sigma=dp\) it is both solenoidal and potential, and hence must be identically zero. Therefore, \((X+\lambda V)(\operatorname{Id}-i\mathcal{H})u=0\), and thus, \((\operatorname{Id}-i\mathcal{H})u=0\) since \((\operatorname{Id}-i\mathcal{H})u|_{\partial(SM)}=0.\)
\end{proof}

The same method yields the antiholomorphic case as well:
\begin{proposition} Let \((M,g,\alpha)\) be a simple magnetic system with \(\operatorname{dim}(M)=2.\) Let \(f\) be a smooth antiholomorphic function on \(SM\). Suppose that \(u\in C^{\infty}(SM)\) satisfies \[(X+\lambda V)u=-f,\ \ \ u|_{\partial(SM)}=0.\] Then \(u\) is antiholomorphic and \(u_{0}=0.\)
\end{proposition}

We also have the following generalizations of results from \cite{paternain211}:
\begin{proposition} Let \((M,g,\alpha)\) be a simple magnetic system with \(\operatorname{dim}(M)=2.\) Suppose that \(u\in C^{\infty}(SM)\) satisfies \[(X+\lambda V)u=-f,\ \ \ u|_{\partial(SM)}=0.\] If \(m\geq 0\) and if \(f\in C^{\infty}(SM)\) is such that \(f_{k}=0\) for \(k\leq -m-1\), then \(u_{k}=0\) for \(k\leq -m\).\label{more}
\end{proposition}
\begin{proof} Choose an arbitrary nonvanishing function \(h\in\Omega_{m}\) and define a degree \(1\) function on \(SM\) via: \[A:=-h^{-1}(X+\lambda V)h.\] Now, note that \(hf\) is holomorphic, and \(hu\) solves: \[(X+\lambda V+A)(hu)=-hf,\ \ \ \ hu|_{\partial(SM)}=0.\] Now by Theorem~\ref{thetheorem} there exists a holomorphic integrating factor \(\omega\in C^{\infty}(SM)\) for \(A\), that is, \((X+\lambda V)\omega=A.\) Furthermore, \[(X+\lambda V)(e^{\omega}hu)=-e^{\omega}hf,\ \ \ \ e^{\omega}hu|_{\partial(SM)}=0.\] The right hand side of the above equation is holomorphic, so by Proposition~\ref{solutionhol} we know that \(e^{\omega}hu\) is holomorphic, and \((e^{\omega}hu)_{0}=0.\) Therefore, \(hu=e^{-\omega}e^{\omega}hu\) is holomorphic and \((hu)_{0}=0.\) Thus, \(u_{k}=0\) for \(k\leq-m.\)
\end{proof}

Again, the proof of the next result is analogous to the one above.
\begin{proposition} Let \((M,g,\alpha)\) be a simple magnetic system with \(\operatorname{dim}(M)=2.\) Suppose that \(u\in C^{\infty}(SM)\) satisfies \[(X+\lambda V)u=-f,\ \ \ u|_{\partial(SM)}=0.\] If \(m\geq 0\) and if \(f\in C^{\infty}(SM)\) is such that \(f_{k}=0\) for \(k\geq m+1\), then \(u_{k}=0\) for \(k\geq m\).\label{oncemore}
\end{proposition}

\begin{remark} It should be understood that the above method implies that the hypotheses that our connection is unitary and Higgs field is skew-Hermitian in Theorem~\ref{mainresult} are superfluous when \(n=1\). The theorem holds for an arbitrary connection and arbitrary Higgs field. Indeed, suppose one has \(u\in C^{\infty}(SM,\mathbb{C})\) with \(u|_{\partial(SM)}=0\),  such that 
\[(X+\lambda V+A'+\Phi)u=-F-\sigma\]
where \(F\in C^{\infty}(M)\) and \(\sigma\) is a 1-form. Simply define \(A'':=-h^{-1}(X+\lambda V)h\) for an arbitrary nonvanishing \(h\in\Omega_{1}.\) Now run the proofs of Proposition~\ref{more} and Proposition~\ref{oncemore} using \(A''+A'+\Phi\) in place of \(A\) above. This yields that \(u\) is both holomorphic and antiholomorphic, and our theorem follows.
\end{remark}

These results immediately yield the following theorem which will turn out to be the key ingredient in the resolution of the tensor tomography question. 
\begin{theorem} Let \((M,g,\alpha)\) be a simple magnetic system with \(\operatorname{dim}(M)=2.\) Suppose that \(u\in C^{\infty}(SM)\) satisfies \[(X+\lambda V)u=-f,\ \ \ u|_{\partial(SM)}=0.\] If \(f\in C^{\infty}(SM)\) has degree \(m\geq 1\), then \(u\) has degree \(m-1\). If \(f\) has degree \(0\), then \(u=0\).\label{degreedowntheorem}
\end{theorem}

Note that (\ref{onelasttime}) may be written \[(X+\lambda V)u = \text{sum} 
\begin{pmatrix}
  \hat{d^{s}} & 0 & \cdots & 0 \\
  \lambda V(\hat{\cdot}) & \hat{d^{s}} & \cdots & 0 \\
  \vdots  & \vdots  & \ddots & \vdots  \\
  0 & \cdots & \lambda V(\hat{\cdot}) & \hat{d^{s}}
\end{pmatrix}
\begin{pmatrix}
  u_{k-1}    \\
  u_{k-2}    \\
  \vdots    \\
  u_{0} 
\end{pmatrix}\] where \(u=\hat{u}_{k-1}+...+\hat{u}_{0}\), and \(u_{i}\) is a symmetric i-tensor. The notation is meant to indicate that we sum over all the entries in the matrix after the matrix multiplication is performed. The operator \(\hat{d^{s}}\) acts on \(u_{i}\) by taking the symmetric derivative, then associating to the resulting tensor a function on \(SM.\) 
 
\begin{proof}[Proof of Theorem~\ref{tensortomography}] Suppose \(I(\sum_{i=0}^{k}\hat{f}_{i})=0,\) and define \[u(x,v):=\int^{\tau(x,v)}_{0}(\sum_{i=0}^{k}\hat{f}_{i})(\varphi_{t}(x,v))dt,\ \ \ (x,v)\in SM.\] Then, \(u|_{\partial(SM)}=0,\) and by Proposition~\ref{theproposition} we know that \(u\in C^{\infty}(SM).\) Since \(\sum_{i=0}^{k}\hat{f}_{i}\) has degree \(k\), by Theorem~\ref{degreedowntheorem} we conclude that \(u\) has degree \(k-1\). Now decompose \(u\) into its Fourier components, and to the components in \(\Omega_{i}\) and \(\Omega_{-i}\) associate a symmetric i-tensor, denoted by \(u_{i}\). Note there is a bijective correspondence between certain smooth functions on \(M\) and symmetric i-tensors which makes this possible, the explicit mapping is described in \cite{paternain211}. This yields our result.
\end{proof}

\begin{remark} We also add that it is immediate that if we are given \(f_{i}\) symmetric i-tensor fields on \(M\) for \(1\leq i\leq k,\) such that 
\[(\sum_{i=0}^{k}\hat{f}_{i}) = (X+\lambda V)u\]
for \(u\in C^{\infty}(SM)\) with \(u|_{\partial(SM)}=0\), then \(I(\sum_{i=0}^{k}\hat{f}_{i})=0\) because
\[I(\sum_{i=0}^{k}\hat{f}_{i})=\int^{\tau(x,v)}_{0}(X+\lambda V)(u)(\varphi_{t}(x,v))dt\] 
which vanishes since we have that \(u|_{\partial{SM}}=0\).
\end{remark}

\end{document}